\documentclass[graybox]{svmult}

\usepackage{amsmath}
\usepackage{amssymb}
\usepackage{psfrag}
\usepackage{epsfig}
\usepackage{mathptmx}       
\usepackage{helvet}         
\usepackage{courier}        
\usepackage{type1cm}        
\usepackage{makeidx}         
\usepackage{graphicx}        
\usepackage{multicol}        
\usepackage[bottom]{footmisc}

\makeindex             

\begin{document}

\title*{Lower Error Bounds for Randomized Multilevel
and Changing Dimension Algorithms}
\author{Michael Gnewuch}
\institute{Michael Gnewuch, School of Mathematics and Statistics,
University of New South Wales, Sydney, NSW, 2052, Australia,
\email{michael.gnewuch@unsw.edu.au}}
%
%
\maketitle

\abstract{We provide lower error bounds for randomized algorithms that approximate integrals of functions depending on an unrestricted or even infinite number of variables. More precisely, we consider the infinite-dimensional integration problem on weighted Hilbert spaces with an underlying anchored decomposition and arbitrary weights. We focus on randomized algorithms and the randomized 
worst case error. 
We study two
cost models for function evaluation which depend on the number of active variables of the chosen sample points. Multilevel algorithms behave very well with respect to the first cost model, while changing dimension algorithms and also dimension-wise quadrature methods, which are based on a similar idea, can take advantage of the more generous second cost model. 
We prove the first non-trivial lower error bounds for randomized algorithms in these cost models and demonstrate their quality in the case of product weights. 
In particular, we show that the randomized changing dimension algorithms provided in [L.~Plaskota, G.~W.~Wasilkowski, J.~Complexity 27 (2011), 505--518] 
achieve convergence rates arbitrarily close to the optimal convergence rate.}

\section{Introduction}

Integrals over functions with an unbounded or infinite number of variables are 
important in physics, quantum chemistry or in quantitative finance, see, e.g.,
\cite{Gil08a, WW96} and the references therein. In the last few years a large amount
of research was dedicated to design new algorithms as, e.g., multilevel and changing dimension
algorithms or dimension-wise quadrature methods, to approximate such integrals efficiently.
Multilevel algorithms were introduced by Heinrich in \cite{Hei98} in the context of integral equations and by Giles in \cite{Gil08a} in the context of stochastic differential equations.
Changing dimension algorithms were introduced by Kuo et al. in \cite{KSWW10} in the context of infinite-dimensional integration on weighted Hilbert spaces and dimension-wise quadrature 
methods were introduced by Griebel and Holtz in \cite{GH10} for multivariate integration;
changing dimension and dimension-wise quadrature algorithms are based on a similar idea.

Here we want to study the complexity of numerical integration on a weighted Hilbert
space of functions with infinitely many variables as it has been done in 
\cite{HW01, KSWW10, HMNR10, NHMR11, Gne10, PW11, B10, BG12, DG12}. The Hilbert space we consider here allows for a so-called
anchored function space decomposition. 
For a motivation of this specific function space setting and connections to problems in stochastics
and mathematical finance see, e.g., \cite{HMNR10,NHMR11}. 
We derive lower error bounds for randomized algorithms
to solve the infinite-dimensional integration problem. Notice that the complexity of integration problems is less well understood in the randomized setting than in the deterministic setting (where
only deterministic algorithms are permitted and the deterministic worst case error is considered), see, e.g., the comments in \cite[p.~487]{NW10}.

Our error bounds are for the randomized worst case error and are expressed in terms
of the cost of a randomized algorithm. Here we solely take account of function evaluations, i.e., the
cost of function sampling, and disregard other cost as, e.g., combinatorial cost. Notice that this makes the statements of our lower bounds only stronger. To evaluate the cost of sampling, we 
consider two sampling models: the nested subspace sampling model (introduced in \cite{CDMR09}, where it was called variable subspace sampling model) and the unrestricted subspace sampling model
(introduced in \cite{KSWW10}). Our lower error bounds are the first non-trivial lower
bounds in these settings, cf. also the comments in the introductions of \cite{HMNR10, PW11}. Due to space restrictions, we do not provide new constructive upper error bounds. 
For the same reason we refer for a formal definition of multilevel algorithms and changing dimension algorithms for the
infinite-dimensional integration problem on weighted Hilbert spaces to 
\cite{HMNR10, NHMR11, Gne10} and \cite{KSWW10,PW11}, respectively.
In this article we only compare our lower bounds to already known upper bounds. In particular, we show that
the randomized changing dimension algorithms provided  for product weights in \cite{PW11}
achieve convergence rates arbitrarily close to the optimal rate of convergence.

Let us mention that similar general lower error bounds for 
infinite-dimensional integration on weighted Hilbert spaces are provided in 
\cite{DG12} in the determistic setting for the anchored decomposition and in \cite{BG12}
in the randomized setting for underlying ANOVA-type decompositions (to treat the latter decompositions, a technically more involved analysis is necessary).

The article is organized as follows: In Section \ref{TGS} the setting we want to study is
introduced. In Section \ref{LB} we prove new lower bounds for the complexity of randomized algorithms for solving the infinite-dimensional integration problem on weighted Hilbert spaces.
In Section \ref{WEIGHTS} we provide the most general form of our lower bounds which is valid 
for arbitrary weights. In Section \ref{SPECIFIC} we state the simplified form of our lower
bounds for specific classes of weights. In particular, we show in Section \ref{PRODUCT} that 
the randomized changing dimension algorithms from \cite{PW11} are essentially optimal.

\section{The general setting}
\label{TGS}

In this section we describe the precise setting we want to study. 
A comparison with the (slightly different) settings described in the 
papers \cite{Gne10,KSWW10,PW11} will be provided in the forthcoming paper \cite{GMR12}; we refer 
to the same paper and to \cite{HW01, HMNR10} for rigorous proofs of the statements on the 
function spaces we consider here.

\subsection{Notation}

For $n\in\mathbb{N}$ we denote the set $\{1,\ldots,n\}$ by $[n]$.
If $u$ is a finite set, then its size is denoted by $|u|$.
We put $\mathcal{U} := \{u\subset \mathbb{N} \,|\, |u|<\infty \}$.
We use  the common Landau symbol $O$, and additionally for non-negative 
functions
$f,g: [0,\infty) \to [0,\infty)$ the notation 
$f=\Omega(g)$ if 
$g= O(f)$.

\subsection{The function spaces}

As spaces of integrands of infinitely many variables, we consider
\emph{reproducing kernel Hilbert spaces};
our standard reference for these spaces is \cite{Aro50}.

We start with univariate functions. Let $D\subseteq \mathbb{R}$ be a Borel
measurable set of
$\mathbb{R}$ and let $k:D\times D\to \mathbb{R}$ be a measurable reproducing kernel
with anchor $a\in D$, i.e., $k(a,a) = 0$. This implies
$k(\cdot,a) \equiv 0$. We assume that $k$ is non-trivial,
i.e., $k\neq 0$.
We denote the reproducing kernel Hilbert space with kernel $k$ by
$H = H(k)$ and its scalar product and norm by $\langle \cdot, \cdot \rangle_H$ and
$\|\cdot\|_H$, respectively. Additionally, we denote its norm unit ball
by $B(k)$. We use corresponding notation for other 
reproducing kernel Hilbert spaces. If $g$ is a constant function in $H$,
then the reproducing property implies 
$g = g(a) = \langle g, k(\cdot,a) \rangle_H = 0$.
Let $\rho$ be a probability measure on $D$. We assume that 
\begin{equation}
\label{cond-M}
 M:= \int_D k(x,x) \,\rho({\rm d}x) <\infty.
\end{equation}   
For arbitrary ${\bf x},{\bf y} \in D^\mathbb{N}$ and $u\in \mathcal{U}$ we define
\begin{equation*}
k_u({\bf x},{\bf y}) := \prod_{j\in u} k(x_j, y_j),
\end{equation*}
where by convention $k_{\emptyset} \equiv 1$.
The Hilbert space with reproducing kernel $k_u$ will be denoted
by $H_u = H(k_u)$.
Its functions depend only on the coordinates $j\in u$. If it is convenient for us,
we identify $H_u$ with the space of functions defined on $D^u$ determined by
the kernel $\prod_{j\in u} k(x_j, y_j)$, and write $f_u({\bf x}_u)$ instead of $f_u({\bf x})$
for $f_u\in H_u$, ${\bf x}\in D^{\mathbb{N}}$, and ${\bf x}_u := (x_j)_{j\in u} \in D^u$.
For all $f_u\in H_u$ and ${\bf x}\in D^{\mathbb{N}}$ we have
\begin{equation}
\label{vanish}
f_u({\bf x}) = 0
\hspace{3ex}\text{if $x_j=a$ for some $j\in u$.}
\end{equation}
This property yields an \emph{anchored decomposition} of functions,
see, e.g., \cite{KSWW10a}.

Let now ${\bf \gamma} = (\gamma_u)_{u\in \mathcal{U}}$ be weights, i.e.,
a family of non-negative numbers. We assume 
\begin{equation}
 \label{summable}
\sum_{u\in\mathcal{U}} \gamma_u M^{|u|} <\infty.
\end{equation}
Let us define the domain $\mathcal{X}$ of functions of infinitely many variables by
\begin{equation*}
\mathcal{X}:= \left\{{\bf x} \in D^{\mathbb{N}} \,\Bigg|\, \sum_{u\in \mathcal{U}} \gamma_u 
k_u(x_j,x_j) <\infty \right\}.
\end{equation*}
Let $\mu$ be the infinite product probability measure of $\rho$ on $D^{\mathbb{N}}$.
Due to our assumptions we have $\mu(\mathcal{X}) = 1$, see \cite[Lemma~1]{HMNR10}
or \cite{GMR12}.
We define
\begin{equation*}
K_{{\bf \gamma}}({\bf x},{\bf y}) := \sum_{u\in\mathcal{U}} \gamma_u k_u({\bf x},{\bf y})
\hspace{2ex}\text{for all ${\bf x},{\bf y} \in\mathcal{X}$.}
\end{equation*}
$K_{{\bf \gamma}}$ is well-defined and, since $K_{{\bf \gamma}}$ is symmetric and positive
semi-definite, it is a reproducing
kernel on $\mathcal{X} \times \mathcal{X}$, see \cite{Aro50}. We denote the corresponding
reproducing kernel Hilbert space by $\mathcal{H}_{{\bf \gamma}} = H(K_{{\bf \gamma}})$, its norm by $\|\cdot\|_{{\bf \gamma}}$, and its norm unit ball by 
$B_{{\bf \gamma}}=B(K_{{\bf \gamma}})$.
For the next lemma see \cite[Cor.~5]{HW01} or \cite{GMR12}.

\begin{lemma}
\label{Lemma6}
The space $\mathcal{H}_{{\bf \gamma}}$ consists of all functions
$f=\sum_{u\in\mathcal{U}} f_u$, $f_u\in H_u$, that have a finite norm
\begin{equation*}
\|f\|_{{\bf \gamma}} = 
\left(\sum_{u\in\mathcal{U}} \gamma^{-1}_u \|f_u\|^2_{H_u} \right)^{1/2}.
\end{equation*}
\end{lemma}

For $u\in \mathcal{U}$ let $P_u$ denote the orthogonal projection
$P_u:\mathcal{H}_{{\bf \gamma}} \to H_u$, $f\mapsto f_u$ onto $H_u$. Then each $f\in \mathcal{H}_{{\bf \gamma}}$
has a unique representation
\begin{equation*}
f = \sum_{u\in \mathcal{U}} f_u
\hspace{2ex}\text{with $f_u = P_u(f) \in H_u$, $u\in \mathcal{U}$.}
\end{equation*}

\subsection{Infinite-dimensional integration}
\label{IDI}

For a given $f\in\mathcal{H}_{{\bf \gamma}}$ we want
to approximate the integral
\begin{equation*}
I(f) := \int_{\mathcal{X}} f({\bf x})\,\mu({\rm d}{\bf x}).
\end{equation*}
Due to (\ref{summable}), $I$ is continuous on
$\mathcal{H}_{{\bf \gamma}}$ and its representer $h\in \mathcal{H}_{{\bf \gamma}}$ is given by
\begin{equation*}
h({\bf x}) = \int_{\mathcal{X}} K_{{\bf \gamma}}({\bf x},{\bf y}) \mu({\rm d}{\bf y}).
\end{equation*}
The operator norm of the integration functional $I$ is given by
\begin{equation}
\label{bedingung}
\|I\|^2_{{\mathcal{H}_{{\bf \gamma}}}} = \|h\|_{{\bf \gamma}}^2 
= \sum_{u\in \mathcal{U}} \gamma_u C^{|u|}_0 <\infty,
\end{equation}
where 
\begin{equation*}
 C_0:= \int_D\int_D
k(x,y)\,\rho({\rm d}x)\,\rho({\rm d}y).
\end{equation*}
We have $C_0 \le M$. 
We assume that $I$ is non-trivial, i.e.,
that $C_0>0$ and $\gamma_u >0$ for at least one $u\in\mathcal{U}$.
For $u\in\mathcal{U}$ and $f\in\mathcal{H}_{{\bf \gamma}}$ we define $I_u := I\circ P_u$,
i.e., 
\begin{equation*}
I_u(f) = \int_{D^u} f_u({\bf x}_u)  \,\rho^u({\rm d}{\bf x}_u),
\end{equation*}
and the representer $h_u$ of $I_u$ in $\mathcal{H}_{{\bf \gamma}}$ is given by
$h_u({\bf x}) = P_u(h)({\bf x})$.
Thus we have
\begin{equation*}
h = \sum_{u\in\mathcal{U}}  h_u
\hspace{2ex}\text{and}\hspace{2ex}
I(f) = \sum_{u\in\mathcal{U}} I_u(f_u)
\hspace{2ex}\text{for all $f\in\mathcal{H}_{{\bf \gamma}}$.}
\end{equation*}
Furthermore,
\begin{equation}
 \label{hu-norm}
\|h_u\|^2_{{\bf \gamma}} = \gamma_{u}C^{|u|}_0
\hspace{2ex}\text{for all $u\in\mathcal{U}$.}
\end{equation}

\subsection{Randomized algorithms, cost models, and errors}

As in \cite{HMNR10}, we assume that algorithms for approximation of
$I(f)$ have access to the function $f$ via a subroutine (``oracle'')
that provides values $f({\bf x})$ for points ${\bf x}\in D^{\mathbb{N}}$.
For convenience we define $f({\bf x}) = 0$ for ${\bf x}\in D^{\mathbb{N}}\setminus \mathcal{X}$.

We now present the cost models introduced in \cite{CDMR09}
and \cite{KSWW10}: In both models we only consider the cost of function 
evaluations. To define the cost of a function evaluation, we fix a
monotone increasing function $\$(\nu): \mathbb{N}_0 \to [1,\infty]$.
For our lower error bounds we will later assume that $\$(\nu) = \Omega(\nu^s)$
for some $s\ge 0$.
For each $v\in\mathcal{U}$ we define the finite-dimensional affine subspace
$\mathcal{X}_{v,a}$ of $\mathcal{X}$ by 
\begin{equation*}
 \mathcal{X}_{v,a} := \{ {\bf x} \in D^\mathbb{N} \,|\, x_j = a \hspace{1ex}\text{for all}
\hspace{1ex} j\in\mathbb{N}\setminus v\}.
\end{equation*}

In the \emph{nested subspace sampling model} function evaluations can
be done in a sequence of affine subspaces 
 $\mathcal{X}_{v_1,a} \subset \mathcal{X}_{v_2,a} \subset \cdots$
for a strictly increasing sequence ${\rm v} = (v_i)_{i\in\mathbb{N}}$ of sets
$\emptyset \neq v_i \in \mathcal{U}$,
and the cost for each function
evaluation is given by the cost function 
\begin{equation}
\label{varcost}
 c_{{\rm v},a}({\bf x}) := \inf\{ \$(|v_i|) \,|\, {\bf x}\in \mathcal{X}_{v_i,a} \},
\end{equation}
where we use the standard convention that $\inf \emptyset = \infty$.
Let $C_{{\rm nest}}$ denote the set of all cost functions of the
form (\ref{varcost}). The nested subspace sampling model was introduced
in \cite{CDMR09}, where it was actually called ``variable subspace sampling
model''. We prefer the name ``nested subspace sampling model'' to clearly
distinguish this model from the following cost model, which is even more
``variable'':

In the \emph{unrestricted subspace sampling model} we are allowed to sample
in any subspace $\mathcal{X}_{u,a}$, $u\in\mathcal{U}$, without any restriction. The
cost for each function evaluation is given by the cost function
\begin{equation}
\label{unrcost}
 c_{a}({\bf x}) := \inf\{ \$(|u|) \,|\, {\bf x}\in \mathcal{X}_{u,a},\, u\in\mathcal{U} \}.
\end{equation}
Put $C_{{\rm unr}} := \{c_a\}$.
The unrestricted subspace sampling model was introduced in
\cite{KSWW10}, where it did not get a particular name.
Obviously, the unrestricted subspace sampling model is more generous
than the nested subspace sampling model. 


We consider randomized algorithms for integration of functions $f\in \mathcal{H}_{{\bf \gamma}}$.
For a formal definition we refer to \cite{CDMR09,Nov88, TWW88, Was89}.
Here we require that a randomized algorithm $Q$ yields for each $f\in\mathcal{H}_{{\bf \gamma}}$
a square-integrable random variable $Q(f)$. (More precisely, a randomized
algorithm $Q$ is a map $Q:\Omega \times \mathcal{H}_{{\bf \gamma}} \to \mathbb{R}$, 
$(\omega,f) \mapsto Q(\omega,f)$, where $\Omega$ is some suitable 
probability space. But for convenience we will not specify the underlying
probability space $\Omega$ and suppress any reference to $\Omega$ or
$\omega\in \Omega$. We use this convention also for other random variables.)  
Furthermore, we require that the cost of a randomized algorithm $Q$, which is defined to be the sum of the cost of all function
evaluations, is a random variable, which may depend on the function $f$. 
That is why we denote this random
variable by ${\rm cost}_c(Q,f)$, $c$ the relevant cost function from $C_{{\rm nest}}$
or $C_{{\rm unr}}$.

We denote the class of all randomized algorithms for numerical integration
on $\mathcal{H}_{{\bf \gamma}}$ that satisfy the very mild requirements stated above by $\mathcal{A}^{{\rm ran}}$. For unrestricted subspace sampling we additionally
consider a subclass $\mathcal{A}^{{\rm res}}$ of $\mathcal{A}^{{\rm ran}}$. We say that an algorithm
$Q\in\mathcal{A}^{{\rm ran}}$ is in $\mathcal{A}^{{\rm res}}$ if there exist an $n\in\mathbb{N}_0$ and sets
$v_1,\ldots,v_n\in\mathcal{U}$ such that for every $f\in\mathcal{H}_{{\bf \gamma}}$ the algorithms $Q$ performs
exactly $n$ function
evaluations of $f$, where the $i$th sample point is taken from $\mathcal{X}_{v_i,a}$, and
$\mathbb{E}({\rm cost}_{c_a}(Q,f)) = \sum^n_{i=1}\$(|v_i|)$.
If additionally $|v_1|, \ldots, |v_n| \le \omega$ for some $\omega \in\mathbb{N}$, we
say that $Q\in\mathcal{A}^{{\rm res}-\omega}$.
Notice that the classes $\mathcal{A}^{{\rm ran}}$, $\mathcal{A}^{{\rm res}}$, and $\mathcal{A}^{{\rm res}-\omega}$ contain in particular non-linear and adaptive algorithms.


The \emph{worst case cost} of a randomized algorithm $Q$ on a class of integrands $F$ is 
\begin{equation*}
 {\rm cost}_{{\rm nest}}(Q,F) := \inf_{c\in C_{{\rm nest}}} \sup_{f\in F}
\mathbb{E}({\rm cost}_c(Q,f))
\end{equation*}
in the nested subspace sampling model and 
\begin{equation*}
 {\rm cost}_{{\rm unr}}(Q,F) := \sup_{f\in F} \mathbb{E}({\rm cost}_{c_a}(Q,f))
\end{equation*}
in the unrestricted subspace sampling model.
The \emph{randomized (worst case) error} $e(Q,F)$ of approximating the integration functional $I$ by $Q$ on $F$ is defined as
\begin{displaymath}
e(Q,F) := \bigg(\sup_{f \in F} \mathbb{E} \left( \left( I(f) - Q(f) \right)^2 \right) \bigg)^{1/2}\, .
\end{displaymath}
For $N\in\mathbb{R}$, ${\rm mod} \in \{ {\rm nest}, {\rm unr}\}$, and $*\in \{{\rm ran},{\rm res}, {\rm res}-\omega\}$ 
let us define the corresponding \emph{$N$th minimal error} by
\begin{equation*}
 e^*_{{\rm mod}}(N,F) := \inf\{ e(Q,F) \,|\, Q\in\mathcal{A}^*
\hspace{1ex}\text{and}\hspace{1ex}{\rm cost}_{{\rm mod}}(Q,F) \le N\}.
\end{equation*}

\subsection{Strong Polynomial tractability}
\label{tractabilitysection}
Let $\omega\in\mathbb{N}$, ${\rm mod} \in \{ {\rm nest}, {\rm unr}\}$, and $*\in \{{\rm ran},{\rm res}, {\rm res}-\omega\}$.
The  $\varepsilon$-complexity of the infinite-dimensional integration
problem $I$ on $\mathcal{H}_{{\bf \gamma}}$ in the considered cost model with respect to the
class of admissable randomized algorithms $\mathcal{A}^*$ is the 
minimal cost among all admissable algorithms, whose randomized
errors are at most $\varepsilon$, i.e.,
\begin{equation}
\label{comp}
  {\rm comp}^*_{{\rm mod}}(\varepsilon, B_{{\bf \gamma}})
 \,:=\, \inf\left\{{\rm cost_{{\rm mod}}}(Q, B_{{\bf \gamma}}) \,|\,
Q \in\mathcal{A}^* \hspace{1ex}\text{and}
\hspace{1ex} e(Q, B_{{\bf \gamma}})\le\varepsilon\right\}.
\end{equation}
The integration problem $I$ is said to be {\em strongly polynomially
tractable} if there are non-negative constants $C$ and $p$ such that
\begin{equation}
\label{pol-tr}
    {\rm comp}^*_{{\rm mod}}(\varepsilon, B_{{\bf \gamma}})\le C \,\varepsilon^{-p} \qquad
   \mbox{for all $\varepsilon>0$}.
\end{equation}
The {\em exponent of strong polynomial tractability} is given by
\begin{equation*}
p^*_{{\rm mod}}= p^*_{{\rm mod}}({\bf \gamma}) := \inf\{ p\,|\, \text{$p$ satisfies \eqref{pol-tr} for some $C>0$} \}.
\end{equation*}
Essentially, $1/p^*_{{\rm mod}}$ is the
\emph{convergence rate} of the
$N$th minimal error
$e^*_{{\rm mod}}(N, B_{{\bf \gamma}})$.
In particular, we have for all $p>p^*_{{\rm mod}}$ that
$e^*_{{\rm mod}}(N, B_{{\bf \gamma}}) = O(N^{-1/p})$.

\section{Lower Bounds}
\label{LB}

We start in Section \ref{WEIGHTS} by proving lower bounds for general 
weights. In Section \ref{SPECIFIC} we show how these bounds simplify
for several specific classes of weights.

\subsection{Results for General Weights}
\label{WEIGHTS}

Let ${\bf \gamma}=(\gamma_u)_{u\in \mathcal{U}}$
be a given family of weights that satisfy (\ref{summable}).
We denote by $\widehat{{\bf \gamma}}$
the family of weights defined by
\begin{equation}
\label{gammahut}
\widehat{\gamma}_u := \gamma_u C_0^{|u|}
\hspace{2ex}\text{for all $u\in \mathcal{U}$.}
\end{equation}
Recall that (\ref{summable}) implies 
$\sum_{u\in \mathcal{U}} \widehat{\gamma}_u < \infty$.
Weights ${\bf \gamma}$ are called \emph{finite-order weights
of order $\omega$}
if  there exists an
$\omega\in \mathbb{N}$ such that
$\gamma_{u}=0$ for all $u\in\mathcal{U}$ with $|u|>\omega$.
Finite-order weights were introduced in \cite{DSWW06} for spaces
of functions with a finite number of variables.
The following definition is taken from \cite{Gne10}.

\begin{definition}
\label{Cut-Off}
For weights ${\bf \gamma}$ and $\sigma \in\mathbb{N}$ let us define the \emph{cut-off weights}
of order $\sigma$
\begin{equation}
\label{gammasigma}
{\bf \gamma}^{(\sigma)}
= (\gamma_u^{(\sigma)})_{u\in\mathcal{U}}
\hspace{2ex}\text{via}\hspace{2ex}
\gamma^{(\sigma)}_{u} =
\begin{cases}
\,\gamma_u
\hspace{2ex}&\text{if $|u| \le \sigma$},\\
\,0
\hspace{2ex} &\text{otherwise.}
\end{cases}
\end{equation}
\end{definition}

Cut-off weights of order $\sigma$ are in particular
finite-order weights of order $\sigma$.
Let us denote by $u_1(\sigma), u_2(\sigma),\ldots $,
the distinct non-empty sets $u\in\mathcal{U}$ with
$\gamma_u^{(\sigma)} >0$
for which
$\widehat{\gamma}_{u_1(\sigma)}^{(\sigma)} \ge
\widehat{\gamma}_{u_2(\sigma)}^{(\sigma)} \ge \cdots$.
Let us put $u_0(\sigma) := \emptyset$. We can make the
same definitions for
$\sigma = \infty$; then we have obviously
${\bf \gamma}^{(\infty)} = {\bf \gamma}$.
For convenience we will often suppress any reference to $\sigma$
in the case where $\sigma = \infty$.
For $\sigma\in\mathbb{N}\cup\{\infty\}$ let us define
\begin{equation*}
{\rm decay}_{{\bf \gamma},\sigma} :=
\sup \left\{ p\in \mathbb{R} \,\Big|\, \lim_{j\to\infty}
\widehat{\gamma}_{u_j(\sigma)}^{(\sigma)}j^p =0
\right\}.
\end{equation*}
Due to assumption (\ref{summable}) the weights we consider always 
satisfy ${\rm decay}_{{\bf \gamma}, \sigma} \ge 1$ for all 
$\sigma\in\mathbb{N}\cup\{\infty\}$.
The following definition is from \cite{Gne10}.

\begin{definition}
\label{Tsternsigma}
For $\sigma\in\mathbb{N}\cup\{\infty\}$ let $t^*_\sigma \in [0,\infty]$
be defined as
\begin{equation*}
t^*_\sigma := \inf \big\{t\ge 0\,|\, \,\exists\, C_t>0 \,\,\forall
\,v \in \mathcal{U}: |\{i\in \mathbb{N} \,|\,
u_i(\sigma) \subseteq v\}| \le C_t|v|^t \big\}.
\end{equation*}
\end{definition}

Let $\sigma\in\mathbb{N}$. Since $|u_i(\sigma)|\le \sigma$ for all $i\in\mathbb{N}$, we have
obviously $t^*_\sigma \le \sigma$.
On the other hand, if we have an infinite sequence
$(u_j(\sigma))_{j\in\mathbb{N}}$, it is not hard to verify that
$t^*_\sigma \ge 1$, see \cite{Gne10}.

For $v_1,\ldots, v_n \in \mathcal{U}$ we use the short hand $\{v_i\}$ for $(v_i)^n_{i=1}$.
Put $v:=\cup^n_{i=1} v_i$
and define the mapping
\begin{equation}
\label{psi-decomp}
\Psi_{\{v_i\},a} := \sum_{j;\, \exists i\in [n]: u_j \subseteq v_i}
P_{u_j}.
\end{equation} 
$\Psi_{\{v_i\},a}$ is the orthogonal projection
of $\mathcal{H}_{{\bf \gamma}}$ onto the subspace
\begin{equation*}
H_{\{v_i\},a} := \sum_{j;\, \exists i\in [n]: u_j \subseteq v_i}
H_{u_j}.
\end{equation*}
Put
\begin{equation*}
 {\rm b}_{\{v_i\},a} := \sup_{f\in B_{{\bf \gamma}}} |I(f) - I(\Psi_{\{v_i\},a}f)|.
\end{equation*}
In the case where $n=1$ and $v=v_1$, we simply write $\Psi_{v,a}$ and
${\rm b}_{v,a}$. In that case we have, due to (\ref{vanish}),
\begin{equation}
 \label{psiva}
(\Psi_{v,a}(f))({\bf x}) = f({\bf x}_v;{\bf a})
\hspace{2ex}\text{for all $f\in \mathcal{H}_{{\bf \gamma}}$
and ${\bf x}\in\mathcal{X}$,}
\end{equation}
where the $j$th component of $({\bf x}_v;{\bf a})$ is defined by
\begin{equation*}
({\bf x}_v;{\bf a})_j :=
\begin{cases}
\, x_j
\hspace{2ex}&\text{if $j\in v$},\\
\, a
\hspace{2ex}&\text{otherwise.}
\end{cases}
\end{equation*}

\begin{lemma}\label{Lemma2.7''}
 Let $v_1,\ldots, v_n\in\mathcal{U}$. Then
\begin{equation*}
 {\rm b}^2_{\{v_i\},a} 
= \sum_{j; \,\forall i\in [n]: u_j \nsubseteq v_i} \widehat{\gamma}_{u_j}.
\end{equation*}
\end{lemma}

\begin{proof}
Let $h_{\{v_i\},a}$ denote the representer of the continuous 
functional $I\circ  \Psi_{\{v_i\},a}$. Due to (\ref{psi-decomp})
we get
\begin{equation*}
  h_{\{v_i\},a}
= \sum_{j;\, \exists i\in [n]: u_j \subseteq v_i} h_{u_j}.
\end{equation*}
Since $h-h_{\{v_i\},a}$ is the representer of $I - I\circ  \Psi_{\{v_i\},a}$
in $\mathcal{H}_{{\bf \gamma}}$, we obtain with (\ref{hu-norm})
\begin{equation*}
  {\rm b}^2_{\{v_i\},a} = \|h- h_{\{v_i\},a}\|^2_{{\bf \gamma}}
= \left\| \sum_{j;\, \forall i\in [n]: u_j \nsubseteq v_i} 
h_{u_j} \right\|^2_{{\bf \gamma}}
= \sum_{j;\, \forall i\in [n]: u_j \nsubseteq v_i} 
\left\| h_{u_j} \right\|^2_{{\bf \gamma}}
= \sum_{j;\, \forall i\in [n]: u_j \nsubseteq v_i}  \widehat{\gamma}_{u_j}.
\end{equation*}
\end{proof}

\begin{lemma}\label{Lemma3.1''}
Let $\theta\in (1/2,1]$ and $v_1,\ldots,v_n\in\mathcal{U}$. Let the randomized algorithm $Q\in\mathcal{A}^{{\rm ran}}$ satisfy $\mathbb{P}(Q(f)= Q(\Psi_{\{v_i\},a}f)) \ge \theta$
for all $f\in B_{{\bf \gamma}}$. Then
\begin{equation*}
e(Q,B_{{\bf \gamma}})^2 \ge (2\theta -1) {\rm b}^2_{\{v_i\},a}.
\end{equation*}
\end{lemma}

\begin{proof}
 Since $\Psi_{\{v_i\},a}$ is an orthogonal projection, we have for all
$f\in B_{{\bf \gamma}}$ that $g:= f-\Psi_{\{v_i\},a}f \in B_{{\bf \gamma}}$. Furthermore, 
$\Psi_{\{v_i\},a}(g) = \Psi_{\{v_i\},a}(-g)=0$. 
Let $A:=\{Q(g) = Q(-g)\}$. Then $\{Q(g) = Q(\Psi_{\{v_i\},a}g)\} 
\cap \{Q(-g) = Q(\Psi_{\{v_i\},a}(-g))\} \subseteq A$, and hence
$\mathbb{P}(A) \ge 2\theta -1$. Therefore
\begin{equation*}
 \begin{split}
e(Q,B_{{\bf \gamma}})^2 \geq
&\max \left\{ \mathbb{E} \left( \left( I(g) - Q( g) \right)^2 \right) ,
\mathbb{E} \left( \left(  I(-g) - Q(-g) ) \right)^2 \right) \right\}\\ 
\geq
&\max \left\{ \int_A \left( I(g) - Q( g) \right)^2 
\,{\rm d}\mathbb{P},
\int_A \left(  I(-g) - Q(-g) ) \right)^2  \,{\rm d}\mathbb{P} \right\}\\ 
\geq &(2\theta -1) \vert I(g) \vert^2 = (2\theta-1)
|I(f) - I( \Psi_{\{v_i\},a}f)|^2.
 \end{split}
\end{equation*}
Hence 
$e(Q,B_{{\bf \gamma}})^2 \ge (2\theta-1)
\sup_{f\in B_{{\bf \gamma}}} |I(f) - I( \Psi_{\{v_i\},a}f)|^2
= (2\theta -1) {\rm b}^2_{\{v_i\},a}$.
\end{proof}
{\bf Further Assumptions.} 
We assume for the rest of this article that $\$(\nu) = \Omega(\nu^s)$ for some
$s\in (0,\infty)$. Furthermore, we assume that $\gamma_{\{1\}} > 0$
and that there exists an $\alpha >0$ such that for univariate integration
in $H(\gamma_{\{1\}}k)$ the $N$th minimal error satisfies
\begin{equation}
\label{ass-1}
 e^{{\rm ran}}(N, B(\gamma_{\{1\}}k)) = \Omega(N^{-\alpha/2}).
\end{equation}
(Note that in the univariate situation the nested and the unrestricted
subspace sampling models are equal; that is why we suppress the reference
to ${\rm unr}$ or ${\rm nest}$.)
Since $B(\gamma_{\{1\}}k) \subseteq B_{{\bf \gamma}}$, assumption (\ref{ass-1})
implies in particular
\begin{equation}
\label{alpha}
 e^{{\rm ran}}_{{\rm nest}}(N, B_{{\bf \gamma}})  = \Omega(N^{-\alpha/2})
\hspace{2ex}\text{and}\hspace{2ex}
e^{{\rm res}-\omega}_{{\rm unr}}(N, B_{{\bf \gamma}}) \ge e^{{\rm res}}_{{\rm unr}}(N, B_{{\bf \gamma}})
= \Omega(N^{-\alpha/2}).
\end{equation} 

\begin{theorem}\label{Theorem3.3''}
Consider the nested subspace sampling model. 
To achieve strong polynomial tractability for the class $\mathcal{A}^{{\rm ran}}$
it is necessary that the weights satisfy 
\begin{equation}
\label{decay-bed}
{\rm decay}_{{\bf \gamma},\sigma} > 1 
\hspace{2ex}\text{for all $\sigma\in\mathbb{N}$.}
\end{equation}
If (\ref{decay-bed}) holds, we have
\begin{equation*}
p^{{\rm ran}}_{{\rm nest}} \ge \max \left\{ \frac{2}{\alpha}, \sup_{\sigma\in\mathbb{N}}
\frac{2 s/t^*_\sigma}{{\rm decay}_{{\bf \gamma},\sigma} - 1} \right\}. 
\end{equation*}
\end{theorem}

As we will see in Section \ref{SPECIFIC}, for product weights and 
finite-order weights condition (\ref{decay-bed}) is equivalent to
${\rm decay}_{{\bf \gamma}} = {\rm decay}_{{\bf \gamma},\infty} > 1$.

\begin{proof}
Let $Q\in\mathcal{A}^{{\rm ran}}$  with ${\rm cost}_{{\rm nest}}(Q,B_{{\bf \gamma}}) \le N$. Then there exists
an increasing sequence ${\rm v} = (v_i)_{i\in\mathbb{N}}$, $\emptyset \neq v_i\in \mathcal{U}$,
such that $\mathbb{E}({\rm cost}_{c_{{\rm v},a}}(Q,f)) \le N+1$ for every $f\in B_{{\bf \gamma}}$.
Let $m$ be the largest integer satisfying $\$(|v_m|) \le 4(N+1)$.
This implies for all $f\in B_{{\bf \gamma}}$ that 
$\mathbb{P}(Q(f) = Q(\Psi_{v_m,a}f)) \ge 3/4$, see (\ref{psiva}). 
Due to Lemma \ref{Lemma2.7''} and \ref{Lemma3.1''} we get
\begin{equation*}
 e(Q,B_{{\bf \gamma}})^2 \ge \frac{1}{2} \sum_{j;\, u_j\nsubseteq v_m} \widehat{\gamma}_{u_j}.
\end{equation*}
Let us now assume that ${\bf \gamma}$ are weights of finite order $\omega$. 
Then we get for $t>t^*_\omega$ and a suitable constant $C_t >0$
\begin{equation*}
 \tau_m := |\{ j\,|\, u_j \subseteq v_m \}| \le C_t |v_m|^t
= O(N^{t/s}),
\end{equation*}
since $N = \Omega(|v_m|^s)$. Hence we get for $p_\omega > {\rm decay}_{{\bf \gamma}, \omega} = {\rm decay}_{{\bf \gamma}} \ge 1$
\begin{equation*}
 e(Q,B_{{\bf \gamma}})^2 \ge \frac{1}{2} \sum_{j=\tau_m+1}^\infty 
\widehat{\gamma}_{u_j} = \Omega(\tau_m^{1-p_\omega}) = 
\Omega \left( N^{\frac{t}{s}(1-p_\omega)} \right).
\end{equation*}
For general weights ${\bf \gamma}$, $\sigma\in\mathbb{N}$, and cut-off weights
${\bf \gamma}^{(\sigma)}$ we have $e(Q,B_{{\bf \gamma}}) \ge  e(Q,B_{{\bf \gamma}^{(\sigma)}})$, 
see also \cite[Remark~3.3]{Gne10}. Since the cut-off weights 
${\bf \gamma}^{(\sigma)}$ are weights of finite order $\sigma$, we get for 
all $p_\sigma > {\rm decay}_{{\bf \gamma},\sigma}$ and $t_\sigma > t^*_\sigma$
\begin{equation}
\label{hilfsgleichung}
 e(Q,B_{{\bf \gamma}})^2 
= \Omega \left( N^{\frac{t_\sigma}{s}(1-p_\sigma)} \right).
\end{equation}
Since (\ref{alpha}) holds, the inequality for the exponent of
tractability follows. 

Now assume that the infinite-dimensional integration problem $I$
is strongly polynomially tractable.
Let $\sigma\in\mathbb{N}$. Then we get from inequality (\ref{hilfsgleichung})
that $p_\sigma \ge 1 + 2s/(t_\sigma\, p^{{\rm ran}}_{{\rm nest}})$. Hence
\begin{equation*}
{\rm decay}_{{\bf \gamma},\sigma} \ge 1 + \frac{2s/t^*_\sigma}{p^{{\rm ran}}_{{\rm nest}}}.
\end{equation*}
Thus we have ${\rm decay}_{{\bf \gamma},\sigma} >1$ for all $\sigma\in\mathbb{N}$. 
\end{proof}

\begin{theorem}\label{Theorem3.1''}
Consider the unrestricted subspace sampling model. 
To achieve strong polynomial tractability for the class $\mathcal{A}^{{\rm res}}$
it is necessary that the 
weights satisfy 
\begin{equation*}
{\rm decay}_{{\bf \gamma},\sigma} > 1 
\hspace{2ex}\text{for all $\sigma\in\mathbb{N}$.}
\end{equation*}
If this is the case, we have
\begin{equation*}
p^{{\rm res}}_{{\rm unr}} \ge \max \left\{ \frac{2}{\alpha}, \sup_{\sigma\in\mathbb{N}}
\frac{2\min\{1,s/t^*_\sigma\}}{{\rm decay}_{{\bf \gamma},\sigma} - 1} \right\}. 
\end{equation*}
\end{theorem}

\begin{proof}
Let $Q \in\mathcal{A}^{{\rm res}}$ have
${\rm cost}_{{\rm unr}}(Q,B_{{\bf \gamma}}) \le N$.
Then there exists an $n\in\mathbb{N}$ and coordinate sets $v_1,\ldots,v_n$
such that $Q$ selects randomly $n$ sample points ${\bf x}_1\in \mathcal{X}_{v_1,a},
\ldots ,{\bf x}_n \in\mathcal{X}_{v_n,a}$ and $\sum^n_{i=1} \$(|v_i|) \le N$. 
Since $Q(f) = Q(\Psi_{\{v_i\},a}f)$ for all $f\in B_{{\bf \gamma}}$, we obtain
from Lemma \ref{Lemma2.7''} and \ref{Lemma3.1''} 
\begin{equation*}
 e(Q,B_{{\bf \gamma}})^2 \ge \sum_{j; \,\forall i\in [n]: u_j \nsubseteq v_i} \widehat{\gamma}_{u_j}.
\end{equation*}
Let us first assume that ${\bf \gamma}$ are weights of finite order $\omega$.
Then we get with Jensen's inequality for $t>t^*_\omega$ and suitable
constants $C_t, c >0$
\begin{equation*}
 \begin{split}
|\{ j \,|\, \exists i\in [n]: u_j \subseteq v_i \}|
&\le \sum^n_{i=1} |\{ j\,|\, u_j \subseteq v_i\}| 
\le \sum^n_{i=1} C_t |v_i|^t \\
&\le C_t \left( \sum^n_{i=1} |v_i|^s \right)^{1/\min\{1,s/t\}}
\le C_t (cN)^{1/\min\{1,s/t\}}.
 \end{split}
\end{equation*}
Hence we obtain for $S:= \lceil C_t(cN)^{1/\min\{1,s/t\}} \rceil$ 
and all $p_\omega > {\rm decay}_{{\bf \gamma}, \omega}$
\begin{equation*}
 e(Q,B_{{\bf \gamma}})^2 \ge \sum^\infty_{j=S+1} \widehat{\gamma}_{u_j} 
= \Omega( S^{1-p_\omega} )
= \Omega \left( N^{\frac{1-p_\omega}{\min\{1,s/t\}}} \right).
\end{equation*}
If we have general weights ${\bf \gamma}$, then we obtain for $\sigma \in\mathbb{N}$
and the cut-off weights ${\bf \gamma}^{(\sigma)}$ that 
$e(Q,B(K_{{\bf \gamma}})) \ge e(Q,B(K_{{\bf \gamma}^{(\sigma)}}))$.
From this and (\ref{alpha}) the inequality for the exponent of 
tractability follows.
Similarly as in the proof of Theorem \ref{Theorem3.3''}, the necessity
of condition (\ref{decay-bed}) is easily established.
\end{proof}

\begin{theorem}\label{Theorem3.2''}
Let $\omega\in\mathbb{N}$ be fixed.
We have for the exponent of tractability $p^{{\rm res}-\omega}_{{\rm unr}}$ in the unrestricted subspace sampling setting
\begin{equation*}
p^{{\rm res}-\omega}_{{\rm unr}}\ge \max \left\{ \frac{2}{\alpha}, \sup_{\sigma\in\mathbb{N}}
\frac{2}{{\rm decay}_{{\bf \gamma},\sigma} - 1} \right\}. 
\end{equation*}
\end{theorem}

\begin{proof}
We follow the lines of the proof of Theorem \ref{Theorem3.1''}, and use the
same notation. The difference is that  this time $Q$ selects
randomly $n$ sample points ${\bf x}_1\in \mathcal{X}_{v_1,a}, \ldots, {\bf x}_n\in 
\mathcal{X}_{v_n,a}$,
where $|v_i| \le \omega$ for all $i\in [n]$, and that we therefore can make the
estimate
$|\{j \,|\, \exists i\in [n] \,:\, u_j\subseteq v_i \}| 
\le 2^{\omega} n = O(N)$,
since $\$(|v_i|) \ge 1$ for all $i\in [n]$ by definition of the function
$\$$.
Hence we get this time for $p> {\rm decay}_{{\bf \gamma}}$ 
\begin{equation*}
 e(Q,B_{{\bf \gamma}})^2 \ge \sum_{j=2^{\omega}n+1}^{\infty} \widehat{\gamma}_{u_j} 
= \Omega(N^{1-p}).
\end{equation*}
\end{proof}

A comparison of Theorem \ref{Theorem3.1''} and \ref{Theorem3.2''} indicates that there are cost functions and classes of finite-order weights for which changing dimension algorithms cannot achieve
convergence rates that are arbitrarily close to the optimal 
rate. Let us recall that for weights of finite order $\omega$, changing dimension 
algorithms as defined in \cite[Proof of Thm.~5]{KSWW10} would only use sample points 
from sample spaces $\mathcal{X}_{u,a}$ with $|u|\le \omega$; see also the comment at the
beginning of Section~4 in \cite{KSWW10}.
Examples of such cost functions and finite-order weights would be $\$(k) = \Omega(k^s)$ and 
lexicographically-ordered weights of order $\omega >s$, see 
Section \ref{LEX}. (A similar observation was made for the deterministic setting, see 
\cite[Thm.~3.2 \& Sect.~3.2.3 ]{Gne10}.)

\subsection{Results for Specific Classes of Weights}
\label{SPECIFIC}

Here we consider some example classes of weights and show 
how our bounds from Section \ref{WEIGHTS} simplify in those settings.

\subsubsection{Product weights and finite-product weights}
\label{PRODUCT}

\begin{definition}
Let $(\gamma_j)_{j\in\mathbb{N}}$ be a sequence of non-negative real
numbers satisfying $\gamma_1\ge \gamma_2 \ge \ldots.$ With the
help of this sequence we define for $\omega\in\mathbb{N}\cup\{\infty\}$  weights
${\bf \gamma} = (\gamma_{u})_{u\in\mathcal{U}}$ by
\begin{equation}
\label{gammafpw}
\gamma_{u} =
\begin{cases}
\prod_{j\in u} \gamma_j
\hspace{2ex}&\text{if $|u| \le \omega$},\\
\,0
\hspace{2ex} &\text{otherwise,}
\end{cases}
\end{equation}
where we use the convention that the empty product is $1$.
In the case where $\omega=\infty$, we call such weights \emph{product weights},
in the case where $\omega$ is finite, we call them \emph{finite-product weights of order} $\omega$.
\end{definition}

Product weights were introduced in \cite{SW98}
and have been studied extensively since then.
Finite-product weights were considered in \cite{Gne10}. 
Observe that for $\sigma\in\mathbb{N}$ the cut-off weights ${\bf \gamma}^{(\sigma)}$ 
of product weights ${\bf \gamma}$ are finite-product weights of order
$\sigma$.

Let us assume that ${\bf \gamma}$ are product or finite-product weights.
As shown in \cite[Lemma~3.8]{Gne10}, we have 
\begin{equation}
\label{decay1=sigma}
{\rm decay}_{{\bf \gamma},1} = {\rm decay}_{{\bf \gamma},\sigma}
\hspace{2ex}\text{for all $\sigma\in\mathbb{N} \cup \{\infty\}$.}
\end{equation}
(Actually, \cite[Lemma~3.8]{Gne10} states inequality 
(\ref{decay1=sigma})
only for all $\sigma\in\mathbb{N}$. But the proof provided in \cite{Gne10}
is also valid for the case $\sigma=\infty$.)
In particular, we see that for strong polynomial tractability with
respect to the nested subspace sampling model and the class 
$\mathcal{A}^{{\rm ran}}$ or with respect to the unrestricted subspace sampling model
and the class $\mathcal{A}^{{\rm res}}$ it is necessary that ${\rm decay}_{{\bf \gamma}} = 
{\rm decay}_{{\bf \gamma},\infty} >1$.
Since $t^*_1 = 1$, we obtain from Theorem \ref{Theorem3.3''}, \ref{Theorem3.1''},
and \ref{Theorem3.2''}
\begin{equation}
\label{neslowboupw}
p^{{\rm ran}}_{{\rm nest}} \ge \max \left\{ \frac{2}{\alpha}\,,\,
\frac{2s}{{\rm decay}_{{\bf \gamma},1} - 1} \right\},
\end{equation}
and
\begin{equation}
\label{unrlowboupw}
p^{{\rm res}}_{{\rm unr}} \ge \max \left\{ \frac{2}{\alpha}\,,\,
\frac{2\min\{ 1,s \}}{{\rm decay}_{{\bf \gamma},1} - 1} \right\},
\hspace{3ex}
p^{{\rm res}-\omega}_{{\rm unr}} \ge \max \left\{ \frac{2}{\alpha}\,,\,
\frac{2}{{\rm decay}_{{\bf \gamma},1} - 1} \right\}.
\end{equation}
Note that the bounds for finite-product weights are the same as 
for product weights.

\begin{remark}
For 
\emph{product and and order-dependent (POD) weights 
$(\gamma_u)_{u\in\mathcal{U}}$},
which were recently introduced in \cite{KSS11} and are of the form
\begin{equation*}
\gamma_u = \Gamma_{|u|} \prod_{j\in u}\gamma_j,
\hspace{3ex}\text{where $\gamma_1\ge\gamma_2\ge \cdots \ge 0$, and
$\Gamma_0=\Gamma_1=1$, $\Gamma_2,\Gamma_3,\ldots \ge 0$,}
\end{equation*}
identity (\ref{decay1=sigma}) still holds; for a proof see \cite{DG12}.
Thus (\ref{neslowboupw}) and (\ref{unrlowboupw}) are
also valid for POD weights. Product and finite-product 
weights are, in particular, POD weights.
\end{remark}

Let us assume that there exist constants $c,C,\kappa >0$, 
$\alpha_1 \ge 0$,
and $\alpha_2\in [0,1]$ such that for all $\emptyset \neq u\in\mathcal{U}$ and
all $n\ge 1$ there exist randomized algorithms $Q_{n,u}$ 
using for all $f_u\in H_u$ at most $n$ function values of $f_u$ with
\begin{equation*}
 \mathbb{E} \left( |I_u(f_u) - Q_{n,u}(f_u)|^2 \right) 
\le \frac{cC^{|u|}}{(n+1)^{\kappa}} 
\left( 1 + \frac{\ln(n+1)}{(|u|-1)^{\alpha_2}} \right)^{\alpha_1(|u|-1)^{\alpha_2}}
\|f_u\|^2_{H_u}.
\end{equation*}
Note that necessarily $\kappa \le \alpha$.
Let us further assume that ${\rm decay}_{{\bf \gamma},1} >1$ and the 
cost function $\$$ satisfies $\$(\nu) = O(e^{r\nu})$ for some 
$r\ge 0$.
Plaskota and Wasilkowski proved in \cite{PW11} with the help
randomized changing dimension algorithms that
\begin{equation*}
 p^{{\rm res}}_{{\rm unr}} \le \max \left\{ \frac{2}{\kappa}, 
\frac{2}{{\rm decay}_{{\bf \gamma},1} -1} \right\}.
\end{equation*}
Hence, if $\Omega(\nu) = \$(\nu) = O(e^{r\nu})$ and $\kappa = \alpha$,
our lower bound (\ref{unrlowboupw}) is sharp and the randomized  
algorithms from \cite{PW11} exhibit essentially the optimal convergence rate.

Let us consider a more specific example, namely the case where $D=[0,1]$, $k$ is the Wiener kernel given by $k(x,y) = \min\{x,y\}$, and $\rho$ is the 
restriction of the Lebesgue measure to $D$. In this case the anchor $a$ is zero. 
The space $H(k)$ is the Sobolev space anchored at zero, and its elements are 
the absolutely continuous functions $f$ with $f(0) = 0$ and square-integrable
first weak derivative.
It is known that $\kappa = 3 = \alpha$, see \cite[Ex.~1 and Prop.~3]{WW07} (or \cite[Ex.~2]{PW11}) and 
\cite[Sect.~2.2.9, Prop.~1]{Nov88}.
Thus the upper bound from \cite{PW11} and our lower bound (\ref{unrlowboupw})
establish for 
$\Omega(\nu) = \$(\nu) = O(e^{r\nu})$ that
\begin{equation*}
 p^{{\rm res}}_{{\rm unr}} = \max \left\{ \frac{2}{3}, 
\frac{2}{{\rm decay}_{{\bf \gamma},1} -1} \right\}.
\end{equation*}
For the same specific example Hickernell et al.
showed for the case $\$(\nu) = \Theta(\nu)$ with the help of multilevel Monte Carlo algorithms that
\begin{equation*}
p^{{\rm ran}}_{{\rm nest}} \le \max \left\{ 2, \frac{2}{{\rm decay}_{{\bf \gamma},1} -1} 
\right\} 
\hspace{2ex}\text{for ${\rm decay}_{{\bf \gamma},1} >1$,}
\end{equation*}
see \cite[Cor.~5]{HMNR10}. Hence
\begin{equation*}
p^{{\rm ran}}_{{\rm nest}} = \frac{2}{{\rm decay}_{{\bf \gamma},1} -1} 
\hspace{2ex}\text{for ${\rm decay}_{{\bf \gamma},1} \in (1,2]$.}
\end{equation*}
Similarly as in the deterministic setting \cite{NHMR11, DG12} or 
in the randomized setting with underlying ANOVA-type decomposition
\cite{BG12}, our lower bound for $p^{{\rm ran}}_{{\rm nest}}$ is sharp for 
sufficiently large ${\rm decay}_{{\bf \gamma},1}$. This may be proved by using
multilevel algorithms based on the integration algorithms provided
in \cite[Sect.~4]{WW07} (cf. also \cite[Sect.~3.2]{PW11}) or on scrambled
scrambled quasi-Monte Carlo algorithms similar to those discussed in 
\cite{BD11}, but providing a rigorous proof for this claim is beyond the
scope of this article.  

\subsubsection{Finite-Intersection Weights}

We restate Definition 3.5 from \cite{Gne10}.

\begin{definition}
\label{def-fiw}
Let $\rho \in \mathbb{N}$.
Finite-order weights $(\gamma_{u})_{u\in \mathcal{U}}$ are
called \emph{finite-intersection weights} with \emph{intersection
degree} at most $\rho$ if we have
\begin{equation}
\label{fiw}
|\{v\in\mathcal{U} \, | \, \gamma_v >0 \,,\, u\cap v \neq \emptyset \}| \le 1+\rho
\hspace{2ex}\text{for all $u\in\mathcal{U}$ with $\gamma_u >0$.}
\end{equation}
\end{definition}

For finite-intersection weights of order $\omega$
it was observed in \cite{Gne10} that $t^*_\sigma = 1$ for all $\sigma \in \mathbb{N}$, resulting in the lower bounds (\ref{neslowboupw}) and (\ref{unrlowboupw}) with ${\rm decay}_{{\bf \gamma},1}$ replaced by ${\rm decay}_{{\bf \gamma},\omega}$.

\subsubsection{Lexicographically-ordered weights}
\label{LEX}

To every set $u\subset\mathbb{N}$ with $|u|=\ell$
we may assign  a word $\varphi(u) := i_1i_2\ldots i_\ell$,
where for $j\in [\ell]$
the number $i_j$ is the $j$th-largest element of $u$.
On the set of all finite words over
the alphabet $\mathbb{N}$ we have the natural lexicographical order
$\prec_{{\rm lex}}$, where by convention the empty word should be
the first (or ``smallest'') word.

\begin{definition}
We call weights ${\bf \gamma}$ \emph{lexicographically-ordered weights} of
\emph{order $\omega$} if $\gamma_\emptyset =1$, $\gamma_u >0$ for all
$u\subset \mathbb{N}$ with $|u| \le \omega$, and
\begin{equation*}
\varphi(u_i) \prec_{{\rm lex}} \varphi(u_j)
\hspace{2ex}\text{for all $i,j\in\mathbb{N}$ satisfying $i<j$.}
\end{equation*}
\end{definition}

Lexicographically-ordered weights were introduced in \cite{Gne10}. Their properties
complement the properties of the other classes of weights considered before, see 
\cite{DG12, Gne10} for more information.
For lexicographically-ordered weights of order $\omega$ we have
$t^*_\sigma = \min\{\sigma, \omega\}$. Hence we get from Theorem \ref{Theorem3.3''}, \ref{Theorem3.1''}, and \ref{Theorem3.2''}
the lower bounds
\begin{equation*}
p^{{\rm ran}}_{{\rm nest}} \ge \max \left\{ \frac{2}{\alpha}\,,\,
\frac{2s/\omega}{{\rm decay}_{{\bf \gamma},\omega} - 1} \right\},
\end{equation*}
and
\begin{equation*}
p^{{\rm res}}_{{\rm unr}} \ge \max \left\{ \frac{2}{\alpha}\,,\,
\frac{2\min\{ 1,s/\omega \}}{{\rm decay}_{{\bf \gamma},\omega} - 1} \right\},
\hspace{3ex}
p^{{\rm res}-\omega}_{{\rm unr}}\ge \max \left\{ \frac{2}{\alpha}\,,\,
\frac{2}{{\rm decay}_{{\bf \gamma},\omega} - 1} \right\}.
\end{equation*}
The lower bounds indicate that in the setting where $\omega > s$
and ${\rm decay}_{{\bf \gamma},\omega}$ is only moderate, changing dimension 
algorithms (which are algorithms of the class $\mathcal{A}^{{\rm res}-\omega}$) cannot achieve the optimal rate of convergence and can be 
outperformed by multilevel algorithms (which can exhibit a behavior similar to the lower bound for 
$p^{{\rm ran}}_{{\rm nest}}$ above). For  the 
deterministic setting and the 
Wiener kernel $k(x,y)=\min\{x,y\}$ on $D=[0,1]$ this was rigorously 
proved  in \cite{Gne10} by lower bounds for
changing dimension algorithms and
upper bounds for multilevel algorithms, see 
\cite[Thm.~3.2 \& Sect.~3.2.3 ]{Gne10}.

\subsection*{Acknowledgment}

The author gratefully acknowledges support by the
German Science Foundation (DFG) under
grant GN 91/3-1 and by the Australian Research Council (ARC).

\end{document}